\documentclass[12pt,a4paper]{article}
\usepackage{graphicx}%
\usepackage{amsmath,amsthm,amsfonts,amssymb} 
\usepackage{txfonts} 
\usepackage{epstopdf}

\DeclareMathOperator\GL{GL}%
\DeclareMathOperator\diag{diag}%
\newcommand{\lire}{\li\,\cdot\,,\cdot\,\re}
\newcommand{\li}{\langle}
\newcommand{\re}{\rangle}
\newcommand{\qu}{^{(2)}}

\newcommand{\cA}{{\mathcal A}} 
\newcommand{\cS}{{\mathcal S}}
\newcommand{\cR}{{\mathcal R}}

\newcommand{\bH}{{\mathbb H}}
\newcommand{\bP}{{\mathbb P}}

\newcommand{\LK}{_{(L)}} 
\newcommand{\LFV}{L\otimes_{F} V} 
\newcommand{\VL}{V_{(L)}} 
\newcommand{\LFL}{L\otimes_{F} L} 
\newcommand{\LL}{L_{(L)}} 
\newcommand{\ot}{\otimes} 

\newcommand{\bPLF}{\bP(L_{F})} 
\newcommand{\bPLL}{\bP(L_{(L)})} 

\newcommand{\x}{\times} 

\renewcommand{\phi}{\varphi}

\newtheorem{thm}{Theorem}[section]
\newtheorem{prop}[thm]{Proposition}
\newtheorem{lem}[thm]{Lemma}
{\theoremstyle{definition}
\newtheorem{defi}[thm]{Definition}
\newtheorem{rem}[thm]{Remark}
\newtheorem*{ass}{Global assumption}
}

\begin{document}

\title{A note on Clifford parallelisms in characteristic two}

\author{Hans Havlicek}
\date{}

\maketitle

\begin{abstract}
It is well known that a purely inseparable field extension $L/F$ with some
extra property and degree $[L:F]=4$ determines a Clifford parallelism on the
set of lines of the three-dimensional projective space over $F$. By extending
the ground field of this space from $F$ to $L$, we establish the following
geometric description of such a parallelism in terms of a distinguished
`absolute pencil of lines' of the extended space: Two lines are Clifford
parallel if, and only if, there exists a line of the absolute pencil that meets
both of them.
\end{abstract}

\section{Introduction}

A detailed survey of various old results about Clifford parallel lines in the
three-dimensional elliptic space (over the real numbers) can be found in the
recent article \cite{bett+r-12a}. One such result is a description of Clifford
parallel lines in terms of the complexified elliptic space, and it may be
summarised as follows: The elliptic metric yields a hyperbolic quadric of the
complex projective space; it is known as the `absolute quadric'. Two lines of
the elliptic space are Clifford parallel if, and only if, there exists a line
of the absolute quadric that meets both of them (after complexification). Since
there are two reguli on the absolute quadric, one actually gets two
parallelisms. It is conventional to label them as the `left' and `right'
Clifford parallelism of the elliptic space. An alternative approach uses the
skew field $\bH$ of real quaternions as underlying vector space of the elliptic
space. The norm function of $\bH$ is a quadratic form, and it yields the
elliptic metric. The left and right Clifford parallelism arise from the left
and right multiplication in $\bH$, respectively; see \cite[p.~8]{blasch-60a}.

Any quaternion skew field $L$ with arbitrary characteristic and centre $F$,
say, can be used (as in the classical case) to define a left and a right
Clifford parallelism in the three-dimensional projective space on the
$F$-vector space $L$. This finding from \cite{karz+k+s-73a} was the starting
point for the research in \cite{blunck+p+p-07a} and \cite{blunck+p+p-10a},
where the following was established: If the ground field $F$ is extended in an
appropriate way, then the description from above of Clifford parallel lines in
terms of the two reguli on a hyperbolic quadric basically remains valid.
However, the details are much more involved for an arbitrary quaternion skew
field $L$ than for the real quaternions $\bH$.
\par
According to \cite{karz+k+s-73a}, there is one more kind of Clifford
parallelism in three-dimensional projective spaces over certain fields $F$ of
characteristic two. The algebraic definition of such a parallelism is similar
to what we had in the preceding paragraphs, but now one has to use a field
extension $L/F$ with degree $[L:F]=4$ and such that $a^2\in F$ for all $a\in
L$. So $L/F$ is purely inseparable. Due to the commutativity of $L$, left
parallel lines now are the same as right parallel lines.

There arises the question if also in this remaining case there is a geometric
description of Clifford parallel lines similar to the one from
\cite{blunck+p+p-07a} and \cite{blunck+p+p-10a}. We shall have to take several
steps before we can provide an affirmative answer. First, we extend the ground
field of the projective space (with underlying vector space $L$) from $F$ to
$L$. Next, we find in this extended space a distinguished plane $\Pi$ which
will be called the \emph{absolute plane}. Its points are determined by the
singular vectors of a quasilinear quadratic form. So the points of the absolute
plane correspond, loosely speaking, to the points of the hyperbolic quadric
from above. However, we must not use all lines of this plane to accomplish our
task, but only those passing through a particular \emph{absolute point} $\cA$
of $\Pi$. This gives a single \emph{absolute pencil of lines} which now takes
the part of the two reguli on a hyperbolic quadric. Finally, our main result is
Theorem~\ref{thm:clifford}: Two lines of the three-dimensional projective space
over $F$ are Clifford parallel if, and only if, there exists a line of the
absolute pencil that meets both of them (after extension of the ground field).
\par
For more information about parallelisms in general, we refer to
\cite{bett+r-12a}, \cite{john-03a}, the book \cite{john-10a}, and the
references therein. At one point we shall come across the double space axiom,
which is part of the well established axiomatic description of Clifford
parallelisms; see, among others, \cite{herz-80a}, the survey in
\cite{karz+k-88a}, and \cite{paso-10a}. Notions from geometry and algebra that
are used without further reference can be found, for example, in
\cite{bour-89a}, \cite{brau-76-12}, \cite{dieu-71a}, \cite{hirsch-85},
\cite{hirsch-98}, \cite{lang-93a}, and \cite{segr-61a}.

\section{Algebraic preliminaries}\label{se:prelim}

Let $V$ be a vector space over a field $F$. We shall also write $V_F$ instead
of $V$ in order to clarify the ground field. If $F$ is a subfield of a field
$L$ then $V$ can be extended to a vector space over $L$ by the following well
known construction \cite[p.~277]{bour-89a}: The tensor product $(\LFV)_F$ can
be made into a vector space over $L$ by letting
\begin{equation}\label{eq:skp}
    a\sum_s a_s\ot v_s :=\sum_s(aa_s)\ot v_s \mbox{~~for all~~} a,a_s\in L,\; v_s\in V .
\end{equation}
We use the shorthand $\VL:=(\LFV)_L$ for this $L$-vector space. The canonical
embedding of $V_F$ in $\VL$ is given by $v\mapsto 1\ot v$. Suppose now that $V$
is also an associative $F$-algebra with unit $e$. A multiplication in $\LFV$
can be defined by the formula
\begin{equation}\label{eq:mult}
    \Big(\sum_s a_s\ot v_s \Big) \cdot \Big(\sum_t a_t'\ot v_t' \Big):=
    \sum_{s,t} (a_s a_t')\ot (v_s v_t')  \mbox{~~for all~~}a_s,a_t'\in L,\;v_s,v_t'\in V .
\end{equation}
In this way $\VL$ turns into an associative algebra over $L$ with unit $1\ot
e$, and the embedding from above is a monomorphism
\cite[pp.~433--434]{bour-89a}.

\par

\begin{ass}
From now on let $F$ be a field of characteristic $2$ and let $L$ be an
extension field of $F$ with degree $[L:F]=4$ and such that $a^2\in F$ for all
$a\in L$.
\end{ass}
The field extension $L/F$ is purely inseparable. Clearly, each $a \in L$ is a
zero of the quadratic polynomial $X^2+a^2\in F[X]$, whence $L$ is a
\emph{quadratic} or, in a different terminology, a \emph{kinematic} algebra
over $F$ \cite[p.~423]{karz-73a}. The quadratic form
\begin{equation}\label{eq:quadrat}
    (\,\cdot\,)^2:L_F\to F : y\mapsto y^2
\end{equation}
has no singular vectors. It is the \emph{norm form} of the algebra $L_F$.
Following \cite[p.~150]{arf-41a}, \eqref{eq:quadrat} is a \emph{quasilinear}
quadratic form. This means that \eqref{eq:quadrat} is a semilinear mapping of
the vector space $L_F$ in the vector space $F$ over its subfield $F\qu$ formed
by all squares, with $F\to F\qu : f\mapsto f^2$ as accompanying field
isomorphism; see also \cite[p.~33]{dieu-71a}.
\par
We now specialise the algebra $V_F$ from above to be $L_F$ and obtain the
four-dimensional commutative and associative $L$-algebra $(\LFL)_L=\LL$ with
unit $1\ot 1$. According to \eqref{eq:skp}, scalars from $L$ act on the
\emph{first} factors of pure tensors. So, the scalar multiples of $1\ot 1$
comprise the one-dimensional subspace $L(1\ot 1)=\{x\ot 1\mid x\in L\}$ of
$\LL$. This subspace is a first isomorphic copy of the field $L$ within the
algebra $\LL$. A second copy is given by the subset $\{1\ot y\mid y\in L\}$.
Here the elements of $L$ appear in their role as vectors of $L_F$. None of
these isomorphic copies of $L$ will be identified with $L$.
\par
The multiplication in the field $L$ is an $L$-bilinear mapping $L\times L\to
L:(x,y)\mapsto x y$. Clearly, this mapping is also $F$-bilinear. By the
universal property of the tensor product (for vector spaces over $F$) there is
a unique $F$-linear mapping
\begin{equation}\label{eq:pi}
    \pi : \LFL \to L  \mbox{~~such that~~}  (x\ot y)^\pi=x y \mbox{~~for all~~}x,y\in L .
\end{equation}
Using \eqref{eq:skp} and \eqref{eq:mult}, a straightforward calculation shows
that $\pi: \LL\to L_L$ actually is a surjective homomorphism of unital
$L$-algebras. Since $L$ is a field, this implies already that
\begin{equation}\label{eq:Pi}
    \Pi:=\ker\pi
\end{equation}
is a three-dimensional maximal ideal of the four-dimensional algebra $\LL$, but
we easily can say more.

\begin{lem}\label{lem:local}
The L-algebra $\LL$ is local and quadratic. The ideal of non-invertible
elements of $\LL$ is the kernel $\Pi$ of the homomorphism $\pi$ from
\eqref{eq:pi}.
\end{lem}
\begin{proof}
Any $g\in\LL$ can be written as $g=\sum_s a_s\ot b_s$ with $a_s,b_s\in L$. We
read off from
\begin{equation*}
    g^2 = \Big(\sum_s a_s\ot b_s\Big)^2 =
    \sum_s a_s^2\ot b_s^2 =
    \Big(\sum_s a_s^2 b_s^2\Big) (1\ot 1)
    = \big(g^\pi\big)^2  (1\ot 1)
\end{equation*}
that $g$ is a zero of the polynomial $X^2+\big(g^\pi\big){}^2\in F[X]\subset
L[X]$, whence $\LL$ is quadratic. For $g\notin\Pi$ we obtain
$g^{-1}=(g^\pi)^{-2}g$, whereas any $g\in\Pi$ clearly has no multiplicative
inverse due to $g^2=0$. Thus the ideal $\Pi$ comprises precisely the
non-invertible elements of $\LL$. So, by definition, $\LL$ is a local algebra,
and $\Pi$ has the required property.
\end{proof}

The canonically defined $L$-linear form $\pi$ maps $1\ot y\mapsto y$ for all
$y\in L$. Due to $F\varsubsetneqq L$, the form $\pi$ does not arise as an
extension of an $F$-linear form on $L_F$. However, the \emph{square of $\pi$},
i.e., the \emph{norm form}
\begin{equation}\label{eq:norm}
   \LL \to L : z\mapsto (z^\pi)^2
\end{equation}
is a quasilinear quadratic form extending the quasilinear quadratic form
\eqref{eq:quadrat} from $L_F$ to $\LL$. Indeed, the norm of $1\ot y$ equals
$y^2$ for all $y\in L$. The non-zero vectors of $\Pi$ are precisely the
singular vectors of the norm form \eqref{eq:norm}.

\par
Let us return to multiplication. Any $b\in L$ determines the mapping
\begin{equation}\label{eq:mu_b}
    \mu_b: L\to L : x\mapsto x b ,
\end{equation}
i.e., the multiplication of elements of $L$ by the fixed element $b$. Any such
$\mu_b$ clearly is an $L$-linear mapping $L_L\to L_L$, but below we shall only
make use of its $F$-linearity. Likewise, for any $h\in\LL$ there is an
$L$-linear mapping $\mu_h: \LL\to\LL : z\mapsto zh$. The canonical extension of
$\mu_b:L_F\to L_F$ from $L_F$ to $\LL$ is the \emph{Kronecker product} (or:
\emph{tensor product} \cite[p.~245]{bour-89a}) $\mu_1\ot \mu_b$ which acts on
pure tensors by sending
\begin{equation}\label{eq:multkronecker}
    x\ot y \mapsto   (x1)\ot (yb) = (x\ot y)(1\ot b) .
\end{equation}
So $\mu_1\ot \mu_b=\mu_{1\ot b}$. Note that, contrary to what we had in
\eqref{eq:skp}, the element $b\in L$ acts on the \emph{second} factors of pure
tensors in \eqref{eq:multkronecker}.

\par
At times it will be convenient to use coordinates (which are written as rows).
To this end we first choose $i,j\in L$ such that $1,i,j$ are linearly
independent over $F$. Then
\begin{equation}\label{eq:basisL}
    (1,i,j,k) \mbox{~~with~~} k:=ij
\end{equation}
is a basis of $L_F$ and
\begin{equation}\label{eq:basisLL}
    (1\ot 1, 1\ot i,1\ot j, 1\ot k)
\end{equation}
is a basis of $\LL$. These bases allow us to replace $L_F$ and $\LL$ with $F^4$
and $L^4$, respectively. For example, the coordinate representation of the
homomorphism $\pi$ from \eqref{eq:pi} is the mapping
\begin{equation*}
    L^4 \to L : (z_0,z_1,z_2,z_3)\mapsto z_0+  i z_1+ j z_2+ k z_3,
\end{equation*}
whence the quadratic norm form \eqref{eq:norm} has the representation
\begin{equation*}
    L^4 \to L : (z_0,z_1,z_2,z_3)\mapsto z_0^2+  i^2 z_1^2+ j^2 z_2^2+ k^2 z_3^2 .
\end{equation*}
If we restrict the domain of the last mapping to $F^4$ and replace its codomain
by $F$ then the description of the quadratic norm form \eqref{eq:quadrat} in
terms of coordinates is obtained. When working in $\LL$ it will often be more
appropriate to change from the basis \eqref{eq:basisLL} to another basis of
$\LL$, namely
\begin{equation}\label{eq:idealbasis}
\begin{aligned}
     (1\ot 1,p,q,r) \mbox{~~with~~}&p:=1\ot i + i\ot 1,\\
    &q:=1\ot j+j\ot 1,\\
    &r:=pq=   1\ot k + i\ot j+ j\ot i + k\ot 1.
\end{aligned}
\end{equation}
For example, if $g\in \LL$ has coordinates $(g_0,g_1,g_2,g_3)\in L^4$ with
respect to the basis \eqref{eq:idealbasis} then $g^\pi=g_0$ and $g_0^2$ is the
norm of $g$.
\par
While the elements $p,q,r$ depend on the choice of $i,j$ in the basis
\eqref{eq:basisL}, the span of $r$ has a basis-free meaning:

\begin{lem}\label{lem:annihil}
Let $\cA$ be the annihilator in $\LL$ of the maximal ideal $\Pi$. Upon choosing
an arbitrary basis $(1,i,j,k)$ of $L_F$ as in \eqref{eq:basisL} and by changing
to the associated basis \eqref{eq:idealbasis} of $\LL$, there holds
\begin{equation}\label{eq:A}
    \cA =  L r = L(1\ot k + i\ot j+ j\ot i + k\ot 1).
\end{equation}
\end{lem}

\begin{proof}
We recall from the proof of Lemma~\ref{lem:local} that $z^2=0$ for all
$z\in\Pi$. Thus the elements from \eqref{eq:idealbasis} satisfy $p r = p^2 q =
0$, $q r=q^2 p=0$, and $r^2=0$. From $\Pi=Lp\oplus Lq\oplus Lr$, all elements
of $\Pi$ are annihilated by $r$, and so $L r\subset\cA$. The annihilator of $p$
clearly is a subspace of $\Pi$ containing $p$ and $r$. Due to $pq=r\neq 0$ and
$\dim\Pi=3$, the two-dimensional subspace $L p\oplus L r$ is the annihilator of
$p$. Likewise the annihilator of $q$ equals $Lq\oplus L r$, whence $\cA\subset
(L p\oplus Lr)\cap (Lq\oplus Lr)= L r$, as required.
\end{proof}
Since $\cA$ is an ideal of the commutative ring $\LL$, it may also be written
as the principal ideal $\LL r$, which is generated by the element $r$. Another
description of this ideal is $\cA=\Pi\cdot \Pi=\{zw\mid z,w\in\Pi\}$. We noted
already subsequent to \eqref{eq:norm} that $\Pi$ is the set of vectors of $\LL$
with norm zero, so that $\cA$ is also related to the norm form of $\LL$.

\section{The absolute pencil}

We shall view $L_F$ as the underlying vector space of a projective space $\bPLF
\cong\bP_3(F)$. We adopt the usual geometric terms: Points, lines and planes
are the subspaces of $L_F$ with dimension one, two, and three, respectively.
Incidence is symmetrised inclusion. Likewise, $\LL=(\LFL)_L$ gives rise to a
projective space $\bPLL\cong\bP_3(L)$. The canonical embedding of $\bPLF$ in
$\bPLL$ is given by $Fx\mapsto L(1\ot x)$. Those points of $\bPLL$ that are
images under this embedding are called \emph{$F$-rational}. A subspace of $\LL$
is called \emph{$F$-rational} if it is spanned by its $F$-rational points. If
$T$ is a subspace of $L_F$ then its extension $T\LK$ is an $F$-rational
subspace of the same dimension, and all $F$-rational subspaces of $\LL$ arise
in this way. The projective space $\bPLL$ has two distinguished subspaces that
stem from the algebra $\LL$:
\begin{defi}\label{def:absolute}
We call the ideal $\cA$ from \eqref{eq:A} the \emph{absolute point} and the
ideal $\Pi$ from \eqref{eq:Pi} the \emph{absolute plane} of the projective
space $\bPLL$. The set of lines through $\cA$ that lie in the plane $\Pi$ is
denoted by $[\cA,\Pi]$, and it is called the \emph{absolute pencil}.
\end{defi}

Take notice that here we adopt the phrase `absolute' in analogy to the
conventional terminology for Cayley-Klein geometries (see, for example,
\cite{gier-82a}) and not in its meaning for polarities, where a point is called
`absolute' if it is incident with its polar hyperplane. However, we shall
encounter polarities at the very end of this section, and encourage the reader
to compare our results with recent findings in \cite[3.5]{knarr+s-09a},
\cite[Thm.~6.4]{knarr+s-13a}, and \cite[2.7,~2.8]{knarr+s-14a} about polarities
with a surprisingly small set of `absolute' points.

\begin{prop}\label{prop:F-rational}
The absolute plane $\Pi$ of the extended projective space\/
$\bPLL\cong\bP_3(L)$ has the following properties:
\begin{enumerate}

\item\label{prop:F-rational-1} The absolute plane contains no $F$-rational
    points.

\item\label{prop:F-rational-2} Each point of the absolute plane is incident
    with at most one $F$-rational line.

\item\label{prop:F-rational-3} Let $Fa$ and $Fb$ be distinct points of\/
    $\bPLF\cong\bP_3(F)$ and let $M=Fa\oplus Fb$ be the line joining them.
    Then the $F$-rational line $M\LK$ meets the absolute plane at the point
    $L(a\ot b + b\ot a)$.

\end{enumerate}
\end{prop}

\begin{proof}
\emph{Ad}~\eqref{prop:F-rational-1}. Assume to the contrary that there exists
an $F$-rational point in $\Pi$. Such a point has the form $L(1\ot c)$ with
$c\in L\setminus\{0\}$, whence $(1\ot c)^\pi=c\neq 0$ yields a contradiction.
\par
\emph{Ad}~\eqref{prop:F-rational-2}. Suppose that a point of $\Pi$ were on two
distinct $F$-rational lines, say $M\LK$ and $N\LK$. Thus $M\LK\cap N\LK$ would
be an $F$-rational point of $\Pi$, which is impossible by
\eqref{prop:F-rational-1}.
\par
\emph{Ad}~\eqref{prop:F-rational-3}. Since $a$ and $b$ are linearly independent
in $L_F$, the tensors $a\ot b$ and $b\ot a$ can be extended to a basis of
$(\LFL)_F$. Hence their sum is non-zero, and $L(a\ot b + b\ot a)$ is a point of
$\bPLL$. From $(a\ot b + b\ot a)^\pi=2ab=0$, this point belongs to $\Pi$, and
$a\ot b + b\ot a = a(1\ot b) + b(1\ot a)$ implies that it is on the
$F$-rational line $M\LK$.
\end{proof}

\begin{rem}
From Proposition~\ref{prop:F-rational}, an injective mapping of the set of
lines of $\bP(L_F)$ into the set of points of the absolute plane $\Pi$ is given
by $M\mapsto \Pi\cap M\LK$. Its algebraic description is based on the
alternating $F$-bilinear mapping of $L_F\times L_F$ to $(\LFL)_F$ sending
$(x,y)$ to $x\ot y+y\ot x$. By the universal property of the tensor product,
this bilinear mapping gives rise to the \emph{alternation operator}
\begin{equation*}
    (\LFL)_F\to (\LFL)_F: x\ot y\mapsto x\ot y+y\ot x .
\end{equation*}
The image of this $F$-linear operator can be identified with the exterior
square $\bigwedge^2 L_F $, whence the points of the form $L(x\ot y+y\ot
x)=L(x\wedge y)$ provide a model of the Klein quadric (over $F$) within the
projective plane $\Pi$ (over $L$). A detailed description of this model is not
within the scope of this article.
\end{rem}

The following is taken from \cite[Satz~1]{karz+k+s-73a}:
\begin{defi}\label{def:clifford}
Let $(M,N)$ be a pair of lines of the projective space $\bPLF\cong\bP_3(F)$. We
say that $M$ is \emph{Clifford parallel} (or shortly: \emph{parallel}) to $N$
if there is an element $b\in L\setminus\{0\}$ such that $N=Mb$. In this case we
write $M\parallel N$.
\end{defi}
Due to the commutativity of $L$, the `left' and 'right' parallel relations from
\cite{karz+k+s-73a} coincide here. The relation $\parallel$ defines indeed a
parallelism on $\bPLF$, i.e., for each line $M$ and each point $Fa$ there is a
unique line $N$ with $Fa\subset N\parallel M$. The parallel class of $M$ is
written as $\cS(M)$.
\par
The Clifford parallelism on the line set of $\bPLF$ satisfies the \emph{double
space axiom} \cite[p.~154]{karz+k+s-73a}. In our setting this result reads as
follows: Given lines $M$ and $N$ with a common point, say $Fa$, and arbitrary
points on $M$ and $N$, say $Fb$ and $Fc$, the unique line $M'\parallel M$
through $Fc$ has a point in common with the unique line $N'\parallel N$ through
$Fb$. Let us repeat the easy proof. From $M'=Ma^{-1}c$ and $N'=Na^{-1}b$
follows that $Fd$ with $d:=a^{-1}bc$ is a common point of $M'$ and $N'$. If the
lines $M,M',N,N'$ are mutually distinct then $Fa,Fb,Fc,Fd$ constitute a
tetrahedron, which one might call a \emph{skew parallelogram}. It seems worth
noting that---in analogy to a parallelogram in an affine plane over a field of
characteristic two---also here the remaining two lines $Fa\oplus Fd$ and
$Fb\oplus Fc$ are parallel to each other. The validity of the double space
axiom implies the following result:
\begin{prop}\label{prop:regular}
All parallel classes of the Clifford parallelism on\/ $\bPLF\cong\bP_3(F)$ are
regular spreads.
\end{prop}
\begin{proof}
Let $M,M_1,M_2$ be mutually distinct parallel lines. So there is a unique
regulus, say $\cR$, containing them. Furthermore, there exists a line $N$ in
the opposite regulus of $\cR$. Through each point of $M$ there is a unique line
$N'\parallel N$ and a unique line $N''$ of the opposite regulus of $\cR$. By
the double space axiom, $N'$ meets $M_1$ and $M_2$ so that $N'=N''$.
Consequently, the opposite regulus of $\cR$ consists of mutually parallel
lines. Applying the double space axiom once more yields that all lines of the
regulus $\cR$ are in the parallel class $\cS(M)$.
\end{proof}

Each line of the projective space $\bPLF$ has a unique parallel line, say $K$,
through the point $F\cdot 1=F$. Upon choosing any $i\in K\setminus F$, the line
$K$ takes the form $K=F\oplus Fi$. By our global assumption on the fields $L$
and $F$ from Section~\ref{se:prelim}, we have $i^2\in F$. So $K$ is the
\emph{intermediate field}\/ of $F$ and $L$ that arises from $F$ by adjoining
the element $i$. We denote the field $K$ by $F[i]$ rather than $F(i)$ in order
to avoid confusion with the subspace $Fi$ of $L_F$. Conversely, any
intermediate field $K$ satisfying $F\varsubsetneqq K\varsubsetneqq L$ is a line
through the point $F$, since $[L:K]=4$ forces $[K:F]=2$.
\par
If $K$ is an intermediate field as above then we may view $L_{(K)}:=(K\ot_F
L)_K$ as a vector space over $K$ which extends $L_F$. This vector space will
usually not be treated as a structure in its own right, but as a substructure
of $\LL$. Thereby we utilise that $\LL$ arises from $L_{(K)}$ (up to a
canonical identification) by extending the ground field from $K$ to $L$
\cite[pp.~278--279]{bour-89a}. Those points of the projective space $\bPLL$
that have at least one generating vector in $K\ot_F L$ are named
\emph{$K$-rational}. A subspace of $\LL$ is called \emph{$K$-rational} if it is
spanned by its $K$-rational points. We are now in a position to describe the
parallel class of the line $K=F[i]$ in terms of the absolute plane.

\begin{thm}\label{thm:punkte}
Let $i\in L\setminus F$. Then the following assertions hold:
\begin{enumerate}

\item\label{thm:punkte-1} The absolute plane $\Pi$ of\/
    $\bPLL\cong\bP_3(L)$ contains a unique $F[i]$-rational line, namely the
    line joining the $F[i]$-rational point $L(1\ot i + i\ot 1)$ with the
    absolute point $\cA$.

\item\label{thm:punkte-2} The absolute point $\cA$ is not $F[i]$-rational.

\item\label{thm:punkte-3} A line $M$ of\/ $\bPLF\cong\bP_3(F)$ is Clifford
    parallel to the line $F[i]=F\oplus Fi$ if, and only if, the extended
    line $M\LK$ meets the absolute plane $\Pi$ at an $F[i]$-rational point.
\end{enumerate}
\end{thm}

\begin{proof}
We extend $1,i$ to a basis $(1,i,j,k)$ of $L_F$ as in \eqref{eq:basisL} and
introduce the associated basis $(1\ot 1,p,q,r)$ of $\LL$ from
\eqref{eq:idealbasis}.
\par
\emph{Ad}~\eqref{thm:punkte-1}. A point of $\Pi$ is $F[i]$-rational precisely
when it can be generated by a vector that belongs to the set $\big(F[i]\ot_F
L\big)\cap \Pi$. We claim that
\begin{equation}\label{eq:schnitt}
    \big(F[i]\ot_F L\big)\cap \Pi = \{ (1\ot y) p \mid y\in L \} .
\end{equation}
Since $p$ is in the ideal $\Pi$, so are all elements from the set on the right
hand side of \eqref{eq:schnitt}. According to \eqref{eq:mult}, we have
\begin{equation}\label{eq:schnittexplizit}
   (1\ot y) p = (1\ot y) (1\ot i + i\ot 1) = 1\ot iy + i \ot y
   \mbox{~~for all~~} y\in L ,
\end{equation}
whence the right hand side of \eqref{eq:schnitt} is a subset of $F[i]\ot_F L$.
Conversely, for any $g$ from the set on the left hand side of
\eqref{eq:schnitt} there are $a,b$ in $L$ with $g=1\ot a+i\ot b$. Now
$g\in\Pi=\ker\pi$ yields $a=ib$, and $g=(1\ot b)p$ follows as in
\eqref{eq:schnittexplizit}. This verifies equation \eqref{eq:schnitt}.
\par
We infer from \eqref{eq:schnitt} that the four vectors
\begin{equation}\label{eq:fano}
    \begin{array}{l@{}lclcl}
    (1\ot 1)&p &=& 1\ot i + i\ot 1 &=& p,\\
    (1\ot i)&p &=& i^2\ot 1+i\ot i    &= &ip \\
    (1\ot j)&p &=& 1\ot k + i\ot j &=&jp+r ,\\
    (1\ot k)&p &=& i^2\ot j +i\ot k     &=& i(jp+r)
    \end{array}
\end{equation}
are all in $\big(F[i]\ot_F L\big)\cap \Pi$. The first and the third vector from
\eqref{eq:fano} are linearly independent over $L$, since $p$ and $r$ belong to
the basis \eqref{eq:idealbasis} of $\LL$. Writing $(y_0,y_1,y_2,y_3) \in F^4$
for the coordinates with respect to the basis \eqref{eq:basisL} of an arbitrary
$y\in L$ yields therefore
\begin{equation}\label{eq:sub}
    (1\ot y)p =   (y_0+iy_1)p   + (y_2+iy_3) (jp+r).
\end{equation}
This shows that the $F[i]$-rational points of $\Pi$ comprise an $F[i]$-subline
of the line $Lp\oplus L(jp+r)= \cA\oplus Lp$. So $\cA\oplus Lp$ is the only
$F[i]$-rational line in $\Pi$. Cf.\ also Figure~\ref{abb:fano} below.
\par
\emph{Ad}~\eqref{thm:punkte-2}. Clearly $r=jp+ 1(jp+r)$. This is the only
possibility to write $r$ as a linear combination with coefficients in $L$ of
the (linearly independent) vectors $p$ and $jp+r$ . Thus $j\notin F[i]$ and
\eqref{eq:sub} imply that there is no $y\in L$ such that $(1\ot y)p$ is a
non-zero vector of $Lr$. So the absolute point $\cA=Lr$ is not $F[i]$-rational.
\par
\emph{Ad}~\eqref{thm:punkte-3}. Let $M$ be a line of $\bPLF$. If $M\parallel
F[i]$ then there is a $b\in L\setminus\{0\}$ with $M=F[i]\cdot b$. From
$F[i]=F\oplus Fi$ follows $M=Fb \oplus F(ib)$. By
Proposition~\ref{prop:F-rational}~\eqref{prop:F-rational-3}, the extended line
$M\LK$ meets $\Pi$ at the point $L(b\ot ib + ib\ot b)$. This point is
$F[i]$-rational, because it can be rewritten as $L(1\ot ib + i\ot b)$.
\par
Conversely, suppose that $M\LK\cap\Pi$ is an $F[i]$-rational point, say $Lm$.
By \eqref{eq:schnittexplizit}, we may assume $m=1\ot ib + i\ot b$ for some
$b\in L\setminus\{0\}$.
Proposition~\ref{prop:F-rational}~\eqref{prop:F-rational-3} shows that the
$F$-rational line $\big(F[i]\cdot b\big)\LK$ passes through the point $Lm$.
From Proposition~\ref{prop:F-rational}~\eqref{prop:F-rational-2}, there is
precisely one $F$-rational line through $Lm$. So we obtain $M=F[i]\cdot b$ or,
said differently, $M\parallel F[i]$.
\end{proof}

\begin{rem}
The description of the parallel class $\cS(F[i])$ from Theorem~\ref{thm:punkte}
can be found in the literature in various guises. It is a special case of the
description of the spread that arises from the field extension $F[i]/F$
according to \cite[Theorem~2]{havl-94a}. (This spread in turn yields a pappian
projective plane whose underlying field is isomorphic to the intermediate field
$F[i]$.) Taking into account that $\cS(F[i])$ is a regular spread, our result
is covered by \cite[Theorem~1.2]{beu+ue-93a}. See also \cite{blunck+p-08a} and
\cite{debr-10a} for related work. In order to fully establish the link with
either of the cited articles, it is sufficient to consider the
\emph{intermediate projective space}\/ $\bP(L_{(F[i])})\cong \bP_3(F[i])$. This
space contains the initial space $\bPLF\cong\bP_3(F)$ as a \emph{Baer
subspace}. The $F[i]$-subline of $\cA\oplus Lp$ mentioned in the proof above is
an \emph{indicator set} of the spread $\cS(F[i])$. It constitutes the `visible'
part of the absolute plane within the intermediate projective space, whereas
the absolute point remains entirely `invisible'.
\par
The images of parallel classes under the Klein mapping are described in
\cite[Lemma~1]{havl-94b}: These are elliptic quadrics (intersections of the
Klein quadric by solids) with the following particular property: The tangent
planes of any such quadric have a common line.

\end{rem}
The essential role of the absolute pencil will come into effect in the next
result, where we describe our Clifford parallelism in terms of the extended
space $\bPLL\cong\bP_3(L)$.

\begin{thm}\label{thm:clifford}
Let $M$ and $N$ be lines of the projective space $\bPLF\cong\bP_3(F)$. Then $M$
and $N$ are Clifford parallel if, and only if, there exists a line of the
absolute pencil $[\cA,\Pi]$ that meets the extended lines $M\LK$ and $N\LK$.
\end{thm}

\begin{proof}
First, let us assume $M\parallel N$. So there exists an $i\in L\setminus F$
such that $M\parallel F[i] \parallel N$. By
Theorem~\ref{thm:punkte}~\eqref{thm:punkte-3}, the extended lines $M\LK$ and
$N\LK$ meet the absolute plane $\Pi$ at $F[i]$-rational points. Recall the
notation $p=1\ot i+i\ot 1$ from \eqref{eq:idealbasis}. Since $\cA\oplus Lp$ is
the only $F[i]$-rational line in $\Pi$ according to
Theorem~\ref{thm:punkte}~\eqref{thm:punkte-1}, each of the $F[i]$-rational
points ${M\LK}\cap\Pi$ and ${N\LK}\cap\Pi$ must be incident with the line
$\cA\oplus Lp\in [\cA,\Pi]$.
\par
\begin{figure}[!ht]\unitlength1.2cm
  \centering
  \begin{picture}(3.43,4)\footnotesize
    \put(0,0.5){\includegraphics[height=8\unitlength]{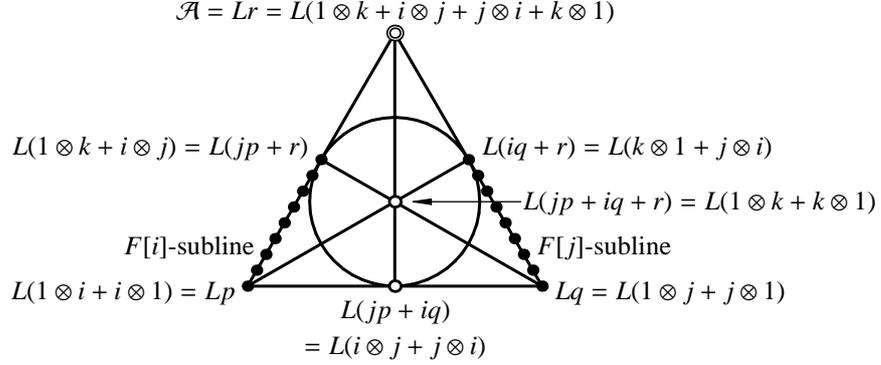}}
    \put(1.75,3.6){\makebox[0cm][c]{$\cA=Lr=L(1\ot k + i\ot j+ j\ot i + k\ot 1)$}}
    \put(2.7,2.1){$L(iq+r) =L(k\ot 1+ j\ot i)$}
    \put(0.8,2.1){\makebox[0cm][r]{$L(1\ot k + i\ot j)=L(jp+r)$}}
    \put(3.15,1.5){$L(jp+iq+r) = L(1\ot k + k\ot 1)$}
    \put( 0.2,1.){\makebox[0cm][r]{$F[i]$-subline}}
    \put(3.3,1.){$F[j]$-subline}
    \put(1.75,0.3){\makebox[0cm][c]{$L(jp+iq)$}}
    \put(1.75,-0.1){\makebox[0cm][c]{$=L(i\ot j+j\ot i)$}}
    \put(0,0.5){\makebox[0cm][r]{$L(1\ot i + i\ot 1)=Lp$}}
    \put(3.5,0.5){$Lq=L(1\ot j + j\ot 1)$}
  \end{picture}
   \caption{Sublines in the absolute plane $\Pi$.}\label{abb:fano}
\end{figure}
Next, we assume $M\not\parallel N$. So there are $i,j\in L\setminus F$ such
that $M\parallel F[i]$ and $N\parallel F[j]$. From $F[i]\not\parallel F[j]$
follows that $1,i,j$ are linearly independent over $F$. We extend these
elements to a basis $(1,i,j,k)$ of $L_F$ as in \eqref{eq:basisL} and introduce
then the basis $(1\ot 1,p,q,r)$ of $\LL$ from \eqref{eq:idealbasis}. By
Theorem~\ref{thm:punkte}~\eqref{thm:punkte-1}, the point ${M\LK}\cap\Pi$
belongs to the subline of $F[i]$-rational points of the line $\cA\oplus Lp$.
Moreover, from Theorem~\ref{thm:punkte}~\eqref{thm:punkte-2}, the absolute
point $\cA$ is not $F[i]$-rational, whence ${M\LK}\cap\Pi\neq\cA$. Hence
$\cA\oplus Lp$ is the \emph{only}\/ line of the absolute pencil $[\cA,\Pi]$
that meets $M\LK$. Exchanging $M$ with $N$ we obtain \emph{mutatis mutandis}:
${N\LK}\cap\Pi$ is an $F[j]$-rational point on the line $\cA\oplus Lq$, the
absolute point $\cA$ is not $F[j]$-rational, and therefore $\cA\oplus Lq$ is
the \emph{only}\/ line of the absolute pencil $[\cA,\Pi]$ that meets $N\LK$.
Since $Lp,Lq,\cA=Lr$ are the vertices of a triangle, there is no line of the
absolute pencil that meets simultaneously the extended lines $M\LK$ and $N\LK$.
See Figure~\ref{abb:fano}, where also two additional points are depicted in
order to obtain a Fano subplane of the absolute plane $\Pi$.
\end{proof}

The approach to the Clifford parallelism in Definition~\ref{def:clifford} makes
use of the group of \emph{Clifford translations}. These are projective
collineations that arise from the multiplication maps $\mu_b$ as in
\eqref{eq:mu_b}, subject to the condition $b\neq 0$. Due to our global
assumption from Section~\ref{se:prelim}, the square of any Clifford translation
is the identical collineation. If $b$ has coordinates $(b_0,b_1,b_2,b_3)\in
F^4$ with respect to an arbitrary basis $(1,i,j,k)$ as in \eqref{eq:basisL}
then the corresponding matrix of $\mu_b$ equals
\begin{equation*}
      \left(\!\!\begin{array}{rrrr}
        b_0         & b_1       & b_2       & b_3 \\
        i^2 b_1     & b_0       & i^2 b_3   & b_2 \\
        j^2 b_2     & j^2 b_3   & b_0       & b_1 \\
        k^2 b_3  & j^2 b_2   & i^2 b_1   & b_0
      \end{array}\!\!\right)  \in \GL_4(F) .
\end{equation*}
The structure of $\mu_b$ becomes more apparent from its extension $\mu_{1\ot
b}$ and by changing to the basis \eqref{eq:idealbasis} which is associated to
\eqref{eq:basisL}. The coordinates of $1\ot b$ with respect to
\eqref{eq:idealbasis} are
\begin{equation*}
    (\underbrace{b_0+b_1i+b_2j+b_3k}_{=\,b},b_1+b_3j,b_2+b_3 i,1)
    =:(b_0',b_1',b_2',b_3')\in L^4 ,
\end{equation*}
and the matrix of $\mu_{1\ot b}$ reads
\begin{equation}\label{eq:multmatrix2}
      \left(\!\!\begin{array}{cccc}
        b'_0         & b'_1       & b'_2       & b'_3 \\
        0           & b'_0       & 0         & b'_2 \\
        0           & 0         & b'_0       & b'_1 \\
        0           & 0         & 0         & b'_0
      \end{array}\!\!\right) \in \GL_4(L).
\end{equation}
The case when $b\neq 0$ is in $F$ does not deserve our interest, since it gives
the identical collineation. Otherwise, we may simplify matters by choosing
w.l.o.g.\ the basis element $i$ equal to the given $b\in L\setminus F$. As a
consequence $b_1=1$ and $b_0=b_2=b_3=0$, which implies that the matrix from
\eqref{eq:multmatrix2} turns into block diagonal form
\begin{equation}\label{eq:multmatrix3}
    \diag\left(
    \begin{pmatrix}i&1\\0&i\end{pmatrix},
    \begin{pmatrix}i&1\\0&i\end{pmatrix}
    \right) \in \GL_4(L).
\end{equation}
From \eqref{eq:multmatrix3} the following observations about the collineation
arising from $\mu_{1\ot i}$ are immediate: The fixed points comprise the line
$\cA\oplus Lp$. A plane is invariant precisely when it contains the line
$\cA\oplus Lp$. The restriction of the collineation to every invariant plane is
an involutory (planar) elation. In particular, the restriction to the absolute
plane $\Pi$ has the absolute point $\cA=Lr$ as its centre.
\begin{rem}
It is straightforward to show (e.g.\ in terms of Pl\"{u}cker coordinates or in
terms of the geometric characterisation from
\cite[vol.~II,~p.~182]{brau-76-12}) that the invariant lines of the
collineation given by \eqref{eq:multmatrix3} constitute a \emph{parabolic
linear congruence}. Furthermore, the $F$-rational lines of this congruence are
exactly the extended lines of the parallel class $\cS(F[i])$.
\end{rem}

Our final aim is to link certain polarities with our Clifford parallelism. Let
$\phi:L\to F$ be an $F$-linear form. Then
\begin{equation}\label{eq:blf}
    \lire_\phi : L\x L\to F : (x,y)\mapsto (xy)^\phi
\end{equation}
is a symmetric $F$-bilinear form satisfying
\begin{equation}\label{eq:konform}
    \li xb,yb\re_\phi=(b^2xy)^\phi=b^2(xy)^\phi = b^2\li x,y\re_\phi
    \mbox{~~for all~~} x,y,b\in L ,
\end{equation}
since $b^2\in F$ holds due to our global assumption from
Section~\ref{se:prelim}. Letting $x=y=1$ in \eqref{eq:konform} shows that the
bilinear form $\lire_\phi$ is \emph{alternating} for $1^\phi=0$. Likewise, the
form turns out to be \emph{anisotropic} for $1^\phi\neq 0$.
\par
From now on let us rule out the zero form $\phi=0$. Then there is a $c\in
L\setminus \ker\phi$, whence for any $a\in L\setminus\{0\}$ we obtain $\li
a,ca^{-1}\re_\phi=c^\phi\neq 0$. So $\lire_\phi$ is non-degenerate and
determines a \emph{projective polarity} $\perp_\phi$ of $\bPLF\cong\bP_3(F)$.
By \eqref{eq:konform}, all Clifford translations commute with this polarity,
which is \emph{null} for $1^\phi=0$ and \emph{elliptic} (i.e., without
self-conjugate points) otherwise. Our null polarities appear (in terms of a
slightly different approach) in \cite[p.~97]{elle+k-61a}. It should also be
noted that our elliptic polarities are \emph{pseudo-polarities} according to
the terminology used in \cite{hirsch-98} and \cite{segr-61a}.

\begin{prop}
For any non-zero $F$-linear form $\phi:L\to F$ the associated polarity
$\perp_\phi$ of\/ $\bPLF\cong\bP_3(F)$ maps every line to a parallel one. If a
line is fixed under $\perp_\phi$ then so are all its parallel lines.
\end{prop}
\begin{proof}
Given a line $M$ there is an $i\in L\setminus F$ and an $a\in L\setminus \{0\}$
with $M=F[i]\cdot a$. For all $y\in M^{\perp_\phi}\setminus\{0\}$ we have $\li
M,y\re_\phi=0$ so that $M \parallel My\subset\ker\phi$. Since $\cS(M)$ is a
spread, there cannot be two distinct lines parallel to $M$ in the plane
$\ker\phi$. Hence, as $y$ varies in the non-empty set
$M^{\perp_\phi}\setminus\{0\}$, the line $My$ remains unchanged. So there is a
constant $c\in L\setminus\{0\}$ such that $My=F[i]\cdot ay = F[i]\cdot c$ holds
for all $y\in M^{\perp_\phi}$. This implies $M^{\perp_\phi}=F[i]\cdot
a^{-1}c\parallel M$.
\par
If $M=M^{\perp_\phi}$ is satisfied for one line $M$ then we obtain
$Mb=M^{\perp_\phi}\cdot b=(Mb)^{\perp_\phi}$ for all $b\in L\setminus\{0\}$ by
\eqref{eq:konform}. So all lines from $\cS(M)$ are fixed under $\perp_\phi$.
\end{proof}
Note that a self-polar line $M$, i.e., a line with the property
$M=M^{\perp_\phi}$, exists precisely when $\perp_\phi$ is a null polarity. If
this is the case then all such lines constitute a \emph{general linear complex
of lines}. On the other hand, any of the elliptic polarities $\perp_\phi$ can
be used to given an alternative definition of the Clifford parallelism
\cite[Remark~3]{havl-94c}.

\begin{rem}
Our Clifford parallelism is readily seen to be \emph{cosymplectic}\/
\cite[Definition~3]{bett+r-05a}, i.e., any two distinct parallel classes
$\cS(M)$ and $\cS(N)$ belong to a common general linear complex of lines. In
order to establish this result, we may assume that both $M$ and $N$ are lines
through the point $F\cdot 1=F$, whence there is a non-zero $F$-linear form
$\phi:L\to F$ that vanishes on the plane $M+N$. The associated polarity
$\perp_\phi$ is null, and its self-polar lines comprise a linear complex which
contains $\cS(M)\cup\cS(N)$. From this observation and from
Proposition~\ref{prop:regular}, our parallelism is also Clifford in the sense
of \cite[Definition~1.9]{bett+r-12a}.
\end{rem}

All bilinear forms from \eqref{eq:blf} can be extended to symmetric
$L$-bilinear forms $\LL\to L$. More generally, any $L$-linear form $\psi:\LL\to
L$ defines a symmetric $L$-bilinear form in analogy to \eqref{eq:blf}. Since
$g^2\in F(1\ot 1)\subset L(1\ot 1)$ holds for all elements $g\in\LL$, the
analogue of \eqref{eq:konform} is satisfied too. However, such a bilinear form
$\lire_\psi$ can be degenerate for $\psi\neq 0$, and we leave a detailed
exposition to the reader. From Lemma~\ref{lem:annihil}, the orthogonal subspace
of the absolute point $\cA$ contains the absolute plane $\Pi$ for any choice of
$\psi$. Therefore, when $\lire_\psi$ is non-degenerate, the projective polarity
$\perp_\psi$ will send the absolute point $\cA$ to the absolute plane $\Pi$,
and the polar planes of the points from $\Pi$ will all contain the absolute
point $\cA$.

\section{Future Research}
We are of the opinion that further investigation should prove worthwhile of
those Clifford parallelisms that arise according to \cite[Satz~1]{karz+k+s-73a}
from purely inseparable field extensions of degree greater than four. This task
should not be confined to the finite-dimensional case. This was one motivation
for avoiding, wherever possible, the use of coordinates in the present paper.
It is striking that, according to the classification from \cite{karz+k+s-73a},
none of those Clifford parallelisms has an analogue when the characteristic of
the ground field is different from two.



\noindent
Hans Havlicek\\
Institut f\"{u}r Diskrete Mathematik und Geometrie\\
Technische Universit\"{a}t\\
Wiedner Hauptstra{\ss}e 8--10/104\\
A-1040 Wien\\
Austria\\
\texttt{havlicek@geometrie.tuwien.ac.at}
\end{document}